\documentclass[12pt]{amsart}

\usepackage{latexsym,amsfonts,amssymb,epsfig,verbatim}
\usepackage{amsmath, latexsym, graphics, textcomp,psfrag}
\usepackage{float}
\usepackage{amsthm}
\usepackage{setspace}

\usepackage[all,2cell,dvips]{xy}

\newcommand{\nc}{\newcommand}

\nc{\dmo}{\DeclareMathOperator}
\nc{\nt}{\newtheorem}

\nt{maintheorem}{Main Theorem}
\nt{theorem}{Theorem}[section]
\nt{lemma}[theorem]{Lemma}
\nt{prop}[theorem]{Proposition}
\nt{question}[theorem]{Question}
\nt{cor}[theorem]{Corollary}
\nt{fact}[theorem]{Fact}
\nt{problem}[theorem]{Problem}
\nt{conjecture}[theorem]{Conjecture}

\dmo{\colim}{colim}

\nc{\Z}{\mathbb{Z}}
\nc{\R}{\mathbb{R}}
\nc{\Q}{\mathbb{Q}}

\nc{\I}{\mathcal{I}}
\nc{\K}{\mathcal{K}}
\nc{\N}{\mathcal{N}}
\nc{\SN}{\mathcal{SN}}
\nc{\SI}{\mathcal{SI}}
\nc{\C}{\mathcal{C}}
\nc{\B}{\mathcal{B}}
\nc{\SB}{\mathcal{SB}}
\nc{\M}{\mathcal{M}}
\nc{\SM}{\mathcal{SM}}
\nc{\F}{\mathcal{F}}

\renewcommand{\S}{\mathcal{S}}

\dmo{\Isom}{Isom}
\dmo{\Homeo}{Homeo}
\dmo{\SHomeo}{SHomeo}
\dmo{\Teich}{Teich}
\dmo{\Mod}{Mod}
\dmo{\Out}{Out}
\dmo{\SMod}{SMod}
\dmo{\PMod}{PMod}
\dmo{\Sp}{Sp}
\dmo{\cd}{cd}
\dmo{\vcd}{vcd}
\dmo{\Stab}{Stab}
\dmo{\Bur}{Bur}

\nc{\margin}[1]{\marginpar{\tiny #1}}

\nc{\p}[1]{\smallskip\noindent{{\bf #1}}}

\include{diagram}

\begin{document}

\input{epsf.sty}


\title{Cohomology of the hyperelliptic Torelli group}

\author{Tara Brendle, Leah Childers, and Dan Margalit}

\address{Tara E. Brendle \\ School of Mathematics \& Statistics \\ University Gardens \\ University of Glasgow \\ G12 8QW \\ tara.brendle@glasgow.ac.uk}

\address{Dan Margalit \\ School of Mathematics\\ Georgia Institute of Technology \\ 686 Cherry St. \\ Atlanta, GA 30332 \\  margalit@math.gatech.edu}

\address{Leah Childers\\ 224 Yates Hall\\Pittsburg State University\\Pittsburg, KS 66762-7502\\lchilder@pittstate.edu}

\thanks{The third author gratefully acknowledges support from the National Science Foundation and the Sloan Foundation.}



\begin{abstract}
Let $\SI(S_g)$ denote the hyperelliptic Torelli group of a closed surface $S_g$ of genus $g$.  This is the subgroup of the mapping class group of $S_g$ consisting of elements that act trivially on $H_1(S_g;\Z)$ and that commute with some fixed hyperelliptic involution of $S_g$.  We prove that the cohomological dimension of $\SI(S_g)$ is $g-1$ when $g \geq 1$.  We also show that $H_{g-1}(\SI(S_g);\Z)$ is infinitely generated when $g \geq 2$.  In particular, $\SI(S_3)$ is not finitely presentable.  Finally, we apply our main results to show that the kernel of the Burau representation of the braid group $B_n$ at $t=-1$ has cohomological dimension equal to the integer part of $n/2$, and it has infinitely generated homology in this top dimension.
\end{abstract}

\maketitle

\vspace{-.5in}

\section{Introduction}

Let $S_g$ denote the closed, connected, orientable surface of genus $g$, and let $s$ be some fixed hyperelliptic involution of $S_g$.  The mapping class group $\Mod(S_g)$ is the group of isotopy classes of orientation-preserving homeomorphisms of $S_g$, and the \emph{hyperelliptic Torelli group} $\SI(S_g)$ is the subgroup of $\Mod(S_g)$ consisting of elements that commute with the homotopy class of $s$ and that act trivially on $H_1(S_g;\Z)$.  The group $\SI(S_g)$ arises, for example, as the fundamental group of the branch locus of the period mapping \cite[Section 4]{hain}.  Also, Ellenberg \cite{je} gives a description of the $\Sp(2g,\Z)$-module structure of the cohomology of the full Torelli group (see below) in terms of the cohomology of $\SI(S_g)$.

\p{Cohomological dimension.} The \emph{cohomological dimension} $\cd(G)$ of a group $G$ is the supremum over all $n$ so that there exists a $G$-module $M$ with $H^n(G;M) \neq 0$.  If a group $G$ has torsion, then $\cd(G) = \infty$.  On the other hand, if $G$ contains a torsion-free subgroup $H$ of finite index, then we can define the \emph{virtual cohomological dimension} $\vcd(G) = \cd(H)$.  It is a theorem of Serre that $\vcd(G)$ is well defined \cite[Th\'eor\`eme 1]{jps}.

\begin{maintheorem}
\label{thm:cd}
For $g \geq 1$, we have $\cd(\SI(S_g)) = g-1$.
\end{maintheorem}

\p{Dimensions of Torelli groups.} Let $\I(S_g)$ denote the \emph{Torelli group} of $S_g$, that is, the subgroup of $\Mod(S_g)$ consisting of elements that act trivially on $H_1(S_g;\Z)$.  Let $\K(S_g)$ denote the subgroup of $\I(S_g)$ generated by Dehn twists about separating simple closed curves.  It is a fact that $\SI(S_g)$ is a subgroup of $\K(S_g)$; this follows immediately from the naturality property of Johnson's homomorphism $\tau$ \cite[Lemma 2D]{djabelian} and Johnson's theorem that $\K(S_g)=\ker(\tau)$ \cite[Theorem 6]{dj2}.

Since
\[ \Mod(S_g) > \I(S_g) \geq \K(S_g) \geq \SI(S_g), \]
it follows from Fact~\ref{cd sub} below that the dimensions of these groups also form a decreasing sequence.  For $g \geq 2$, we in fact have the following:

\onehalfspace
\begin{center}
\begin{tabular}{r c l }
$\vcd(\Mod(S_g))$ & = & $4g-5$ \\
$\cd(\I(S_g))$ & = & $3g-5$ \\
$\cd(\K(S_g))$ & = & $2g-3$ \\
$\cd(\SI(S_g))$ & = & \ \,$g-1$.  
\end{tabular}
\end{center}
\singlespacing

\vspace{-3ex}

The first equality is due to Harer \cite[Theorem 4.1]{jlh}.  An alternate proof was given by Ivanov \cite[Theorem 6.6]{nvi4}.  The lower bound of $4g-5$ was also given by Mess \cite[Proposition 1]{mess}, and the upper bound follows from work of Culler--Vogtmann \cite{cv}.  The inequality $\cd(\I(S_g)) \geq 3g-5$ was proven by Mess \cite[Proposition 1]{mess}, and the inequality $\cd(\I(S_g)) \leq 3g-5$ was proven by Bestvina--Bux--Margalit \cite[Theorem A]{bbm}.  The dimension $\cd(\K(S_g))$ was computed by Bestvina--Bux--Margalit \cite[Theorem B]{bbm}.

In the case $g=2$, the groups $\I(S_2)$, $\K(S_2)$, and $\SI(S_2)$ are all equal (combine \cite[Theorem 8]{bh} with \cite[Theorem 2]{jp}\footnote{Powell states his result for $g \geq 3$, but his proof holds in the case $g=2$.}).  This agrees with the fact that $3g-5$, $2g-3$, and $g-1$ are all equal when $g=2$.

\p{The hyperelliptic Johnson filtration.} The \emph{Johnson filtration} of $\Mod(S_g)$ is the sequence of groups $\N_k(S_g)$ defined by:
\[ \N_k(S_g) = \ker(\Mod(S_g) \to \Out(\pi_1(S_g)/\pi_1^k(S_g))),\]
where $\pi_1^k(S_g)$ is the $k$th term of the lower central series for $\pi_1(S_g)$.  By definition, $\N_1(S_g) = \Mod(S_g)$ and $\N_2(S_g) = \I(S_g)$.  It is a theorem of Johnson that $\N_3(S_g) = \K(S_g)$ \cite{dj2}.  An argument of Farb \cite[Theorem 5.10]{farb} and the fact that $\N_k(S_g) \leq \K(S_g)$ for $k \geq 3$ gives
\[ g - 1 \leq \cd(\N_k(S_g)) \leq 2g-3 \]
for $g \geq 2$ and $k \geq 3$ (see Fact~\ref{cd sub} below).

We may also consider the groups $\SN_k(S_g) = \N_k(S_g) \cap \SMod(S_g)$.  For $k \geq 1$, we have $\SN_k(S_g) \leq \SI(S_g)$, and so $\cd(\SN_k(S_g)) \leq g-1$ for $g \geq 1$ and $k \geq 1$.  On the other hand, we will prove in Proposition~\ref{prop:cd sn lower} below that $\SN_k(S_g)$ contains a subgroup isomorphic to $\Z^{g-1}$ for $g \geq 1$ and $k \geq 1$.  Therefore, we have the following theorem.

\begin{theorem}
\label{thm:sn}
For $g \geq 1$ and $k \geq 1$, we have
\[ \cd(\SN_k(S_g)) = g-1. \]
\end{theorem}

\p{Top-dimensional homology.} Bestvina--Bux--Margalit proved that the top-dimensional homology of $\I(S_g)$ is infinitely generated \cite[Theorem C]{bbm}.  We prove the analogous result for $\SI(S_g)$.

\begin{maintheorem}
\label{thm:htop}
For $g \geq 2$, the group $H_{g-1}(\SI(S_g);\Z)$ is infinitely generated.
\end{maintheorem}

Since $\I(S_1)$ is trivial, Main Theorem~\ref{thm:htop} does not hold for $g=1$.  Mess proved that $\SI(S_2) = \I(S_2)$ is an infinite rank free group \cite[Proposition 4]{gm}, from which it immediately follows that $H_1(\SI(S_2);\Z)$ is infinitely generated.

It is not known in general whether or not the groups $\SI(S_g)$ are finitely generated or finitely presented for $g \geq 3$.  However, we have the following immediate consequence of Main Theorem~\ref{thm:htop}.

\begin{cor}
The group $\SI(S_3)$ is not finitely presentable.
\end{cor}

\p{The Burau representation.} Let $\Bur_{n}$ denote the kernel of the reduced Burau representation at $t=-1$.  In Section~\ref{sec:top}, we explain the precise connection between $\Bur_{n}$ and the hyperelliptic Torelli group.  We obtain the following theorem.

\begin{theorem}
\label{thm:burau}
For $n \geq 5$, we have
\[ \cd(\Bur_{n}) = \left\lfloor \frac{n}{2} \right\rfloor. \]
Also, $H_{\left\lfloor \frac{n}{2} \right\rfloor}(\Bur_{n};\Z)$ is infinitely generated.
\end{theorem}

Our approaches to proving our main theorems are modeled on the arguments of the paper by Bestvina--Bux--Margalit \cite{bbm}.  On the other hand, some of the details are more subtle in the present situation, and we place most of our emphasis on these points.

\p{Acknowledgments.} We would like to thank Joan Birman for comments on an earlier draft.


\section{The complex of symmetric cycles}
\label{sec:sb}

Our main theorems will be proven by analyzing the action of $\SI(S_g)$ on a contractible complex $\SB_x(S_g)$, which we define in this section.  This complex is a symmetric version of the complex of minimizing cycles introduced by Bestvina--Bux--Margalit \cite{bbm}.  

Fix some nonzero $x \in H_1(S_g;\Z)$.  The complex $\SB_x(S_g)$ will be defined as a certain set of isotopy classes of 1-cycles in $S_g$ representing $x$.  The complex does depend on the choice of $x$ (there are finitely many isomorphism types of complexes for the infinitely many choices of $x$), but the main feature of $\SB_x(S_g)$, its contractibility, will not depend on $x$.

A \emph{1--cycle} in $S_g$ is a finite formal sum
\[ \sum k_ic_i \]
where $k_i \in \R$, and each $c_i$ is an oriented simple closed curve in $S_g$;
the set $\{c_i: k_i \neq 0\}$ is called the \emph{support}.  We say
that the 1-cycle is \emph{simple} if the curves of the support are
pairwise disjoint, and we say that it is \emph{positive} if each $k_i$
is positive.  

Let $\S$ denote the set of isotopy classes of oriented simple closed curves in $S_g$.  We may regard the isotopy class of a simple, positive 1--cycle in $S_g$ as an element of $\R_{\geq 0}^\S$, the space of functions $\S \to \R_{\geq 0}$.

A 1-cycle is \emph{skew-symmetric} if its support is fixed as a set by $s$ but has its orientation reversed by $s$.

A \emph{skew-symmetric pair of curves} in $S_g$ is a pair of disjoint, oriented simple closed curves in $S_g$ interchanged and reversed by $s$, that is, a pair of disjoint, oriented curves of the form $\{c,-s(c)\}$.  Both curves in a skew-symmetric pair must be nonseparating.  This follows, for example, from the fact that $s$ acts by $-I$ on $H_1(S_g;\Z)$.

A \emph{skew-symmetric multicurve} in $S_g$ is a nonempty collection of skew-symmetric pairs of curves in $S_g$ that are homotopically nontrivial, pairwise disjoint, and pairwise non-homotopic.  Note that a skew-symmetric multicurve has no connected components that are preserved by $s$.  Also, two simple closed curves lying in a given skew-symmetric multicurve can only be isotopic only if they lie in the same skew-symmetric pair.

A \emph{basic skew-symmetric cycle} is a positive, skew-symmetric 1-cycle
\[ \sum_{i=1}^{n} \frac{k_i}{2}(c_i-s(c_i)) \]
where the support $\{c_i,-s(c_i)\}$ is a skew-symmetric multicurve, and where the $[c_i]$ form a linearly independent subset of $H_1(S_g;\R)$.  

Let $\SM$ denote the set of isotopy classes of skew-symmetric multicurves in $S_g$ that are unions of supports of basic skew-symmetric cycles representing $x$.

Let $M = \{c_1,-s(c_1),\dots,c_m,-s(c_m)\}$ be a skew-symmetric multicurve whose isotopy class $[M]$ lies in $\SM$.  The set
\begin{align*}
P_M = \left \{ (k_1,\dots,k_m) \in \R_{\geq 0}^m :  \sum_{i=1}^m \frac{k_i}{2}(c_i-s(c_i)) \text{ is a skew-symmetric} \right. \\ \left. \text{1-cycle representing } x \right \}  
\end{align*}
is a convex polytope in $\R_{\geq 0}^m$.  Indeed, it is the convex hull of the points corresponding to basic skew-symmetric cycles representing $x$.  The faces of $P_M$ correspond exactly to skew-symmetric multicurves $M' \subseteq M$ with $[M'] \in \SM$.  

The cell complex $\SB_x(S_g)$ is defined as follows: the set of cells is
\[  \{P_M : [M] \in \SM \}. \]
We identify two cells if they are equal in $\R_{\geq 0}^\S$ and endow the quotient with the weak topology.  We refer to $\SB_x(S_g)$ as the \emph{complex of symmetric cycles}.

\begin{theorem}
\label{b contract}
Let $g \geq 1$, and let $x \in H_1(S_g;\Z)$ be any primitive element.  The complex $\SB_x(S_g)$ is contractible.
\end{theorem}

Bestvina--Bux--Margalit studied a complex $\B_x(S_g)$ on which $\SB_x(S_g)$ is modeled.  Theorem~\ref{b contract} can be proven in the same way as the contractibility of $\B_x(S_g)$; see \cite[Theorem E]{bbm} and \cite[Proposition 7]{igen}.  The only thing to check is that their functions $\textrm{Drain}$ and $\textrm{Surger}$ preserve skew-symmetry.  But this is easy to verify.  Thus, we do not repeat the proof.

The quotient map $S_g \to S_g / \langle s \rangle$ is a branched cover of $S_g$ over a sphere $S_{0,2g+2}$ with $2g+2$ \emph{cone points} of order two, namely, the images of the $2g+2$ fixed points of $s$.

For our purposes, the cone points are simply marked points; we only use this terminology to distinguish these $2g+2$ points from other marked points.  When we discuss simple closed curves (and homotopies of curves) in $S_{0,2g+2}$, we treat cone points (and all marked points) as if they are punctures.  So, for instance, curves are not allowed to pass through cone points.

The image of any skew-symmetric multicurve $M$ under the quotient $S_g \to S_{0,2g+2}$ is an unoriented multicurve $\overline M$ in $S_{0,2g+2}$, that is, a collection of essential, pairwise disjoint, pairwise nonhomotopic simple closed curves in $S_{0,2g+2}$.  Let $Z=Z(M)$ denote the number of components of $S_{0,2g+2}-\overline M$ that do not contain any of the $2g+2$ cone points, and let $P=P(M)$ denote the number of components that do contain cone points.

\begin{prop}
\label{cell dim}
For any $[M] \in \SM$, we have
\[ \dim(P_M) = Z.\]
\end{prop}

\begin{proof}

The dimension of the space of positive 1-cycles that represent $x$ and that are supported in $M$ is one fewer than the number of complementary components of $M$ in $S_g$ \cite[Lemma 2.1]{bbm}.  The number of such components is precisely $P+2Z$.  To obtain the dimension of $P_M$, we simply need to impose the condition of skew-symmetry.  This introduces $|\overline M|$ independent equations, namely, that the coefficients on the curves interchanged by $s$ must be equal.  Thus,
\[ \dim(P_M) = P+2Z-1 - |\overline M| = P+2Z-1-(P+Z-1) = Z,\]
as desired.
\end{proof}


\section{The Birman--Hilden theorem}

Let $\SHomeo^+(S_g)$ denote the group of orientation-preserving homeomorphisms of $S_g$ that commute with the hyperelliptic involution $s$.  We define the \emph{hyperelliptic mapping class group} $\SMod(S_g)$ to be the group of isotopy classes of elements of $\SHomeo^+(S_g)$.  We do not, a priori, require the isotopies to be $s$-equivariant.  Thus, $\SMod(S_g)$ is a subgroup of $\Mod(S_g)$.

There is a short exact sequence
\[ 1 \to \langle s \rangle \to \SHomeo^+(S_g) \to \Homeo^+(S_{0,2g+2}) \to 1. \]
This is useful because $S_{0,2g+2}$ is a simpler object than $S_g$.  As such, one would hope for an analogous short exact sequence on the level of mapping class groups.  Birman--Hilden proved that this is indeed the case \cite[Theorem 7]{bh}, that is, for $g \geq 2$, there is a short exact sequence:
\[ 1 \to \langle [s] \rangle \to \SMod(S_g) \to \Mod(S_{0,2g+2}) \to 1.\]
This theorem amounts to the fact that, if an element of $\SHomeo^+(S_g)$ is isotopic to the identity, then it is isotopic to the identity within $\SHomeo^+(S_g)$.

We require a souped-up version.  Let $P$ be a set of $2p$ marked points in $S_g$ and say that $s$ interchanges the points of $P$ in pairs.  Let $\overline P$ denote the image of $P$ in $S_{0,2g+2}$.  Let $\SMod(S_g,P)$ be the set of isotopy classes of orientation-preserving homeomorphisms of $S_g$ that commute with $s$ and preserve the set $P$.  Similarly, define $\Mod(S_{0,2g+2},\overline P)$ as the set of isotopy classes of orientation-preserving homeomorphisms of $S_{0,2g+2}$ that preserve the set of $2g+2$ cone points and preserve the set $\overline P$.

We have the following generalized short exact sequence, also due to Birman--Hilden \cite[Theorem 1]{bhannals}.

\begin{theorem}
\label{thm:bh}
Let $g \geq 1$.  If $g=1$, assume that $p > 0$.  There is a short exact sequence:
\[ 1 \to \langle [s] \rangle \to \SMod(S_g,P) \to \Mod(S_{0,2g+2},\overline P) \to 1. \]
\end{theorem}

Theorem~\ref{thm:bh} does not hold as stated for the case where $g=1$ and $p=0$.  Indeed, consider the element $\phi$ of $\SHomeo^+(T^2)$ that is rotation by $\pi$ in one of the two circle factors.  Let $\overline \phi$ denote the image of $\phi$ in $\Homeo^+(S_{0,4})$.  The mapping class $[\phi]$ is trivial, but the mapping class $[\overline \phi]$ is nontrivial, as it induces a nontrivial permutation of the cone points of $S_{0,4}$.  Thus, we do not have a natural well-defined map $\SMod(T^2) \to \Mod(S_{0,4})$.

We can, however, modify Theorem~\ref{thm:bh} in the case $g=1$, $p=0$.  First of all, each element of $\Mod(T^2)$ has a (linear) representative that commutes with $s$, and so $\SMod(T^2) \cong \Mod(T^2)$.  Second, there is a non-canonical isomorphism $\Mod(T^2) \to \Mod(T^2,p)$, where $p$ is one of the fixed points of $s$.  The reason for this is that each element of $\Mod(T^2)$ has a (linear) representative that fixes the image of the origin under the covering map $\R^2 \to T^2$.

Let $\overline p$ denote the image of $p$ in $S_{0,4}$, and let $\Mod(S_{0,4}, \overline p)$ denote the subgroup of $\Mod(S_{0,4})$ consisting of elements that fix the marked point $\overline p$.  We have the following special case of the Birman--Hilden theorem.

\begin{theorem}
\label{thm:bh special}
There is a short exact sequence:
\[ 1 \to \langle [s] \rangle \to \SMod(T^2) \to \Mod(S_{0,4},\overline p) \to 1. \]
\end{theorem}

Note that, in the statement of Theorem~\ref{thm:bh special}, the group $\Mod(S_{0,4},\overline p)$ is a subgroup of $\Mod(S_{0,4})$, not $\Mod(S_{0,5})$, since $\overline p$ is already a cone point of $S_{0,4}$.


\section{Cohomological dimension}

In this section, we prove Main Theorem~\ref{thm:cd}, which states that $\cd(\SI(S_g)) = g-1$.  We start by showing that $\cd(\SI(S_g)) \geq g-1$ (Proposition~\ref{prop:cd lower}).  

We will use the following fact \cite[Chapter VIII, Proposition 2.4]{ksb}.

\begin{fact}
\label{cd sub}
If $H$ is a subgroup of a group $G$, then $\cd(H) \leq \cd(G)$.
\end{fact}

\begin{prop}
\label{prop:cd lower}
For $g \geq 1$, we have $\cd(\SI(S_g)) \geq g-1$.
\end{prop}

\begin{proof}

We can find a collection of $g-1$ mutually disjoint, essential, homotopically distinct, separating simple closed curves in $S_g$ that are fixed by $s$.  The Dehn twists about these curves generate a subgroup of $\SI(S_g)$ that is isomorphic to $\Z^{g-1}$ \cite[Lemma 3.17]{fm}.  It is a basic fact that $\cd(\Z^n) = n$ for any $n \geq 0$; see \cite[Section VIII.2]{ksb}.  Applying Fact~\ref{cd sub}, we deduce the desired lower bound.
\end{proof}

We now give a variant of Proposition~\ref{prop:cd lower} which, together with our Main Theorem~\ref{thm:cd}, gives Theorem~\ref{thm:sn}.

\begin{prop}
\label{prop:cd sn lower}
For $g \geq 1$ and $k \geq 1$, we have $\cd(\SN_k(S_g)) \geq g-1$.
\end{prop}

\begin{proof}

Let $c_1, \dots, c_{g-1}$ denote the separating simple closed curves from the proof of Proposition~\ref{prop:cd lower}.  We can find disjoint nonseparating simple closed curves $a_1, \dots, a_{g-1}$ fixed by $s$ and with the properties that the geometric intersection numbers $i(a_i,c_i)$ are all equal to 2 and $i(a_i,c_j)=0$ for $i \neq j$.  For each $i$, define $d_i = T_{a_i}(c_i)$, where $T_{a_i}$ is the Dehn twist about $a_i$.  By construction, each $d_i$ is fixed by $s$.  Also, we have $i(c_i,d_i)=4$ and $i(c_i,d_j)=0$ when $i \neq j$.

Fix some $k \geq 1$ and some $i$.  As in Farb's proof of the lower bound $\cd(\N_k(S_g)) \geq g-1$ \cite[Theorem 5.10]{farb}, some element $\gamma_{i,k}$ of the group $\langle T_{c_i} , T_{d_i} \rangle$ lies in $\N_k(S_g)$.  Since each $\gamma_{i,k}$ lies in $\SMod(S_g)$, the group $\langle \gamma_{1,k} , \dots, \gamma_{g-1,k} \rangle$ lies in $\SN_k(S_g)$.  Since the $\gamma_{i,k}$ all have infinite order and are supported on pairwise disjoint subsurfaces of $S_g$, we in fact see that this group is a free abelian group of rank $g-1$.  The proposition now follows from Fact~\ref{cd sub}.
\end{proof}

We now aim to show that $\cd(\SI(S_g)) \leq g-1$ (Proposition~\ref{prop:cd upper}).  Our basic tool is the following fact \cite[Section VIII.2, Exercise 4]{ksb}.

\begin{prop}
\label{prop:quillen}
Suppose that a group $G$ acts on a contractible cell complex $X$.  We have
\[ \cd(G) \leq \sup_{\tau} \{ \cd(\Stab_G(\tau)) + \dim(\tau) \} \]
where the supremum is taken over all cells $\tau$ of $X$.
\end{prop}

Of course, we will apply Proposition~\ref{prop:quillen} to the case of the $\SI(S_g)$ action on the complex of symmetric cycles $\SB_x(S_g)$.

\subsection{The Birman exact sequence and dimension} 

Let $S_{g,n}$ denote a closed, connected, orientable surface of genus $g$ with $n>0$ marked points.  The group $\Mod(S_{g,n})$ is the group of isotopy classes of orientation-preserving homeomorphisms of $S_g$ that preserve the set of $n$ marked points.

Assume $2g+n >3$.  Denote the $n$th marked point of $S_{g,n}$ by $p$, and let $\Mod(S_{g,n},p)$ denote the subgroup of $\Mod(S_{g,n})$ preserving $p$.  There is a natural map $\Mod(S_{g,n},p) \to \Mod(S_{g,n-1})$ obtained by forgetting that $p$ is marked.  The Birman exact sequence \cite[Section 1]{jb} identifies the kernel:
\[ 1 \to \pi_1(S_{g,n-1},p) \to \Mod(S_{g,n},p) \to \Mod(S_{g,n-1}) \to 1. \]
Let $\PMod(S_{g,n})$ denote the subgroup of $\Mod(S_{g,n})$ consisting of elements that induce the trivial permutation of the marked points.  We also have the restriction:
\[ 1 \to \pi_1(S_{g,n-1},p) \to \PMod(S_{g,n}) \to \PMod(S_{g,n-1}) \to 1. \]
We would like to use the Birman exact sequence to gain information about the cohomology of $\Mod(S_{g,n})$ and its subgroups.  The key is the following fact \cite[Chapter VIII, Proposition 2.4]{ksb}.

\begin{fact}
\label{cd ses}
Suppose we have a short exact sequence of groups
\[ 1 \to K \to G \to Q \to 1. \]
Then $\cd(G) \leq \cd(K) + \cd(Q)$.
\end{fact}

\begin{prop}
\label{cd pmod}
For $n \geq 3$ we have
\[ \cd(\PMod(S_{0,n})) \leq n-3. \]
\end{prop}

\begin{proof}

The group $\PMod(S_{0,3})$ is trivial \cite[Proposition 2.3]{fm}, and hence it has cohomological dimension 0.  Since $\pi_1(S_{0,n})$ is a free group, it has cohomological dimension (at most) 1.  The proposition then follows by applying the Birman exact sequence and Fact~\ref{cd ses} inductively.
\end{proof}

In the case of $g \geq 1$, we will require the following more delicate bound on cohomological dimension.  As above, $\PMod(S_g,P)$ is the group of isotopy classes of homeomorphisms of $S_g$ fixing each point in $P$.

\begin{prop}
\label{cd bes}
Let $g \geq 1$.  Suppose $P$ is a set of $p$ pairs of marked points in $S_g$, where the points in each pair are identified by $s$.  Let $H$ be some subgroup of $\SMod(S_g)$ with $[s] \notin H$.  Let $F : \SMod(S_g,P) \cap \PMod(S_g,P) \to \SMod(S_g)$ be the forgetful map, and let $G$ be a subgroup of $F^{-1}(H)$.  Then $\cd(G) \leq \cd(H) + p$.
\end{prop}

\begin{proof}

Let $\overline P$ denote the image of $P$ in $S_{0,2g+2}$.  Since $[s] \notin \PMod(S_g,P)$, the Birman--Hilden theorem (Theorems~\ref{thm:bh} and~\ref{thm:bh special}) implies that the groups $F^{-1}(H)$ and $H$ are identified isomorphically with their images in $\Mod(S_{0,2g+2},\overline P)$ and $\Mod(S_{0,2g+2})$, respectively.  Applying the Birman exact sequence inductively, and using Fact~\ref{cd ses} and the fact that $\cd(\pi_1(S_{0,n})) = \cd(F_{n-1}) = 1$, we obtain $\cd(F^{-1}(H)) \leq \cd(H) + p$.  By Fact~\ref{cd sub}, we have $\cd(G) \leq \cd(F^{-1}(H))$, and the proposition follows.
\end{proof}


\subsection{Dimensions of cell stabilizers}

In this section, we fix some $g \geq 2$ and we fix some skew-symmetric multicurve $M$ with $[M] \in \SM$.  The stabilizer of $[M]$ in $\SI(S_g)$ is exactly the stabilizer of the cell $P_M \subseteq \SB_x(S_g)$ in $\SI(S_g)$.

As above, we denote the image of $M$ in $S_g/\langle s\rangle \cong S_{0,2g+2}$ by $\overline M$.  Say that $S_{0,2g+2}-\overline M$ has $P$ connected components that contain some of the $2g+2$ cone points and $Z$ components that do not contain any cone points.  Denote these subsurfaces by $\overline R_1, \dots, \overline R_P$ and $\overline R_{P+1},\dots, \overline R_{P+Z}$, respectively.

Say that $\overline R_i$ contains $k_i$ cone points and that the preimage $R_i$ of $\overline R_i$ in $S_g$ has genus $g_i$.  Denote the number of components of $\overline M$ in the boundary of $\overline R_i$ by $p_i$, so each $\overline R_i$ is homeomorphic to a sphere with $k_i$ cone points and $p_i$ punctures.  For our purposes, punctures play the same role as marked points.

\begin{lemma}
\label{p}
Let $1 \leq i \leq P$.  Then $R_i$ is homeomorphic to $S_{g_i,2p_i}$, where $g_i = (k_i-2)/2$.
\end{lemma}

\begin{proof}

By the Riemann--Hurwitz formula \cite[Section 7.2.2]{fm}, the orbifold Euler characteristic of $\overline R_i$ is
\[ \chi(\overline R_i) = 2 - p_i - k_i/2. \]
Since orbifold Euler characteristic is multiplicative under orbifold covering maps, we have
\[ \chi(R_i) = 4 - 2p_i - k_i. \]
Now, to each curve of $\overline M$, there corresponds exactly two curves of $M$.  Therefore, $R_i$ has $2p_i$ punctures.  Also, when $k_i >0$, the cover $R_i$ has one connected component.  Plugging the last two facts into the general formula $\chi(S_{g,n}) = 2-2g-n$, we obtain a second formula for the Euler characteristic of $R_i$:
\[ \chi(R_i) = 2-2g_i-2p_i. \]
Combining our two formulas for $\chi(R_i)$, we find that $g_i = (k_i-2)/2$.
\end{proof}

\begin{lemma}
\label{some guy}
We have
\[ \sum_{i=1}^P g_i = g-P+1. \]
\end{lemma}

\begin{proof}

Combining Lemma~\ref{p} with the fact that $\sum k_i = 2g+2$, we have
\[ \sum_{i=1}^P g_i = \sum_{i=1}^P \frac{k_i-2}{2} = \left(\frac{1}{2} \sum_{i=1}^P k_i\right) -P = \frac{2g+2}{2} - P = g-P+1. \]
\end{proof}

\begin{lemma}
\label{other guy}
We have
\[ |\overline M| + 1 = P + Z. \]
\end{lemma}

\begin{proof}

The quantity on the right hand side is the total number of components of $S_{0,2g+2} - \overline M$.  Since $S_{0,2g+2}$ is a sphere, the number of complementary components is $|\overline M|+1$.
\end{proof}

Let $G(M)$ be the free abelian group generated by the Dehn twists in the curves of $M$.  

\begin{lemma}
\label{mtwist}
The group $G(M) \cap \SI(S_g)$ is trivial.
\end{lemma}

\begin{proof}

Because $M$ contains no separating curves (see Section~\ref{sec:sb}), the intersection $G(M) \cap \K(S_g)$ is trivial \cite[Theorem A.1]{bbm}.  As in the introduction, $\SI(S_g) < \K(S_g)$.  The lemma follows.
\end{proof}

\begin{lemma}
\label{lemma:big computation}
Assume that Main Theorem~\ref{thm:cd} is true for all genera between 1 and $g-1$ inclusive.  We have
\[ \cd(\Stab_{\SI(S_g)}(M)) \leq g-1-Z. \]
\end{lemma}

\begin{proof}

There is a short exact sequence
\[ 1 \to G(M) \to \Stab_{\Mod(S_g)}(M) \to \Mod(S_g-M) \to 1 \]
(see \cite[Proposition 3.20]{fm}).  Since $G(M) \cap \SI(S_g)$ is trivial (Lemma~\ref{mtwist}), $\Stab_{\SI(S_g)}(M)$ is isomorphic to its image $G$ in $\SMod(S_g-M)$.

By a theorem of Ivanov, each element of $G$ fixes each $R_i$ and fixes each puncture of each $R_i$ \cite[Theorem 3]{nvi2}.  Thus for each $i$ there is a well-defined map $\Stab_{\SI(S_g)}(M) \to \PMod(R_i)$; denote the image by $G_i$.  The group $G$ is contained in $\prod G_i$.  By Fact~\ref{cd ses}, $\cd(G) \leq \sum \cd(G_i)$.

We claim that
\begin{enumerate}
\item for $1 \leq i \leq P$, we have $\cd(G_i) \leq g_i - 1 + p_i$, and
\item for $P+1 \leq i \leq P+Z$, we have $\cd(G_i) \leq p_i-3$.
\end{enumerate}
We start with the first statement.  If $k_i=2$, then $p_i=1$ and $R_i$ is a sphere with two marked punctures, and so $\PMod(R_i)$ is trivial.  Thus, $G_i$ is trivial.  This means $\cd(G_i)$ is 0, which is certainly less than or equal to $g_i+1-p_i=0$.

Now assume $k_i > 2$, i.e., $g_i > 0$.  By filling in the $p_i$ punctures of $R_i$, we obtain a forgetful map $\PMod(R_i) \cap \SMod(R_i) \to \SMod(S_{g_i})$.  The image of $G_i$ under this map is a subgroup of $\SI(S_{g_i})$ \cite[Lemma 5.10]{bbm}.  By induction, we have $\cd(\SI(S_{g_i})) \leq g_i-1$.  By Proposition~\ref{cd bes}, we have $\cd(G_i) \leq g_i-1+p$.

We now address the second statement, which treats the case where $k_i=0$. The surface $\overline R_i$ is homeomorphic to a sphere with $p_i$ punctures.  Since $G < \SMod(S_g-M)$ and $G$ does not permute components of $S_g-M$, the group $G_i$ is isomorphic to a subgroup of $\PMod(\overline R_i)$.  By Fact~\ref{cd sub} then, $\cd(G_i) \leq \cd(\PMod(\overline R_i))$.  But by Proposition~\ref{cd pmod}, the latter is at most $p_i-3$.

We now have
\begin{eqnarray*}
\cd(\Stab_{\SI(S_g)}(M)) & = & \cd(G) \\
& \leq & \sum \cd(G_i) \\
& \leq  & \sum_{i=1}^P (g_i-1+p_i) + \sum_{i=P+1}^{P+Z} (p_i-3) \\ 
& = & \sum_{i=1}^{P} g_i + \sum_{i=1}^{P+Z} p_i  - P - 3Z \\
& = & (g-P+1) + 2 |\overline M |  - P -3Z \\
& = & g - 1 - Z + 2\left (|\overline M| + 1 - P - Z\right) \\
& = & g - 1 - Z.
\end{eqnarray*}
The first equality and first inequality follow from the above discussion.  The second inequality is the content of the claim.  The third equality follows from Lemma~\ref{some guy} and the fifth equality from Lemma~\ref{other guy}.  The other two equalities are just algebra.
\end{proof}

\subsection{Proof of Main Theorem~\ref{thm:cd}}

Combining Proposition~\ref{cell dim} and Lemma~\ref{lemma:big computation} we obtain the following.

\begin{prop}
\label{stab dim plus cell dim}
Assume that Main Theorem~\ref{thm:cd} is true for all genera between 1 and $g-1$ inclusive.
For any cell $\tau$ in $\SB_x(S_g)$, we have 
\[ \cd(\Stab_{\SI(S_g)}(\tau)) + \dim(\tau) \leq g-1. \]
\end{prop}

We can now obtain the following lower bound for $\cd(\SI(S_g))$ by induction on $g$ and applying Propositions~\ref{prop:quillen} and~\ref{stab dim plus cell dim}

\begin{prop}
\label{prop:cd upper}
For $g \geq 1$, we have $\cd(\SI(S_g)) \leq g-1$.
\end{prop}

Propositions~\ref{prop:cd lower} and~\ref{prop:cd upper} immediately imply Main Theorem~\ref{thm:cd}.


\section{Infinite generation of top homology}
\label{sec:top}

In this section, we prove Main Theorem~\ref{thm:htop}.  The basic strategy is to employ the following fact, which is a consequence of the Cartan--Leray spectral sequence \cite[Fact 8.2]{bbm}.

\begin{prop}
\label{ss}
Suppose a group $G$ acts without rotations on a contractible cell complex X.  Suppose that for each cell $\tau$ of $X$ we have
\[ \cd(\Stab_G(\tau)) + \dim(\tau) \leq D. \]
Then for any vertex $v$ of $X$, the group $H_D(\Stab_G(v);\Z)$ injects into $H_D(G;\Z)$.
\end{prop}

We will apply Proposition~\ref{ss} to the case of the $\SI(S_g)$ action on $\SB_x(S_g)$.  By Proposition~\ref{stab dim plus cell dim}, it suffices to show that the group $H_{g-1}(\Stab_{\SI(S_g)}(v);\Z)$ is infinitely generated for some choice of vertex $v$ of $\SB_x(S_g)$.

We proceed by induction on $g$.  By Mess's theorem that $\I(S_2)$ is an infinite rank free group \cite[Proposition 4]{gm}, Main Theorem~\ref{thm:htop} holds for $g = 2$.  Now assume that $g \geq 3$.

Let $v$ be a vertex of $\SB_x(S_g)$ corresponding to a skew-symmetric nonseparating curve (or, a skew-symmetric pair where the two curves in the pair are homotopic), and let $\Stab_{\SI(S_g)}(v)$ denote the stabilizer of $v$ in $\SI(S_g)$.  There is a splitting
\[ \Stab_{\SI(S_g)}(v) \cong \SI(S_{g-1}) \ltimes K, \]
where $K$ is an infinite rank free group \cite[Theorem 4.11 plus Lemma 5.8]{sibk}.  What is more, $K$ contains a Dehn twist $T_c$, where $c$ is a symmetric separating curve in $S_g$ cutting off a handle containing $v$.  It follows from the explicit description of the splitting that $T_c$ is fixed by the action of $\SI(S_{g-1})$.

Applying the Hochschild--Serre spectral sequence, we obtain
\[ H_{g-1}(\Stab_{\SI(S_g)}(v);\Z) \cong H_{g-2}(\SI(S_{g-1}); H_1(K;\Z)). \]

Johnson defined a homomorphism that maps $\K(S_g)$ to a free abelian group and maps each Dehn twist in $\K(S_g)$ nontrivially \cite[Proposition 1.1]{morita}.  Since $K < \SI(S_g) < \K(S_g)$, it follows that $A = \langle [T_c] \rangle$ is a free submodule of $H_1(K;\Z)$.

Since $A$ is torsion free, the universal coefficients theorem gives us
\[ H_{g-2}(\SI(S_{g-1});A) \cong H_{g-2}(\SI(S_{g-1});\Z)\otimes A. \]
Because $A$ is a trivial $\SI(S_{g-1})$-module, the latter is infinitely generated by induction.  

It thus remains to show that $H_{g-2}(\SI(S_{g-1});A)$ injects into the group $H_{g-2}(\SI(S_{g-1}); H_1(K;\Z))$.  The short exact sequence
\[ 1 \to A \to H_1(K;\Z) \to H_1(K;\Z)/A \to 1 \]
induces a long exact sequence of homology groups:
\[ \cdots \to H_{g-1}(\SI(S_{g-1}); H_1(K;\Z)/A) \to H_{g-2}(\SI(S_{g-1});A) \]
\[ \to H_{g-2}(\SI(S_{g-1});H_1(K;\Z)) \to \cdots. \]
By Main Theorem~\ref{thm:cd}, the first term shown is trivial.

Thus, $H_{g-2}(\SI(S_{g-1}); H_1(K;\Z)) \cong  H_{g-1}(\Stab_{\SI(S_g)}(v);\Z)$ is infinitely generated.  By Propositions~\ref{ss} and~\ref{stab dim plus cell dim}, our Main Theorem~\ref{thm:htop} is proven.


\p{Application to the Burau representation.}
Let $P$ be a pair of points in $S_g$ that are interchanged by $s$.  Let $\SI(S_g,P)$ denote the subgroup of $\SMod(S_g,P)$ consisting of elements that act trivially on the relative homology $H_1(S_g,P;\Z)$.  We have isomorphisms
\begin{eqnarray*}
\Bur_{2g+1} &\cong& \SI(S_g) \times \Z \qquad \qquad \text{ and} \\
\Bur_{2g+2} &\cong& \Stab_{\SI(S_{g+1})}(v) \times \Z
\end{eqnarray*}
when $g \geq 2$; see \cite[Lemma 5.8]{sibk}, \cite{si}, and \cite{mp}.

The group $\Stab_{\SI(S_{g+1})}(v)$ is isomorphic to $\SI(S_{g+1}) \ltimes F_\infty$.  Thus, by Main Theorem~\ref{thm:cd} and Fact~\ref{cd ses}, we have $\cd(\Stab_{\SI(S_{g+1})}(v)) \leq g+1$.  On the other hand, we showed above that $H_{g+1}(\Stab_{\SI(S_{g+1})}(v);\Z)$ is infinitely generated, so in fact $\cd(\Stab_{\SI(S_{g+1})}(v)) = g+1$.  Theorem~\ref{thm:burau} now follows immediately from the K\"unneth formula.

\bibliographystyle{plain}
\bibliography{cdsi}

\end{document}